\documentclass[12pt,english]{article}

\usepackage{amsmath,amsthm,amssymb,calc,eucal, latexsym}
\usepackage[dvips]{graphics, graphicx}

\usepackage[T2A]{fontenc}
\usepackage[english]{babel}
\usepackage{latexsym}
\renewcommand{\ge}{\geqslant}
\renewcommand{\le}{\leqslant}

\newcommand{\floor}[1]{\left[ {#1} \right]}
\newcommand{\ost}[1]{\left\{ {#1} \right\}}
\newcommand{\abs}[1]{\left| {#1} \right|}
\newcommand{\skob}[1]{\left( {#1} \right)}

\newtheorem{Th}{Theorem}
\newtheorem{Lemma}[Th]{Lemma}
\newtheorem*{Hyp}{Conjecture}
\newtheorem{note}{Remark}

\def\F{\Bbb F}

\def\E{\mathsf {E}}

\author{A.~S.~Volostnov}
\title{On some double sums with multiplicative characters
\footnote{
This work was supported by grant Russian Scientific Foundation RSF 14--11--00433.}
}
\date{}

\begin{document}
\maketitle

\bigskip

\begin{center}
    Annotation.
\end{center}

{\it \small
    We obtain a new upper bound for binary sums with multiplicative characters over variables belong to some sets, having small additive doubling.
}

\section*{Introduction}

Let $p$ be a prime number, $\F_p$ be the prime field and $\chi$ be a nontrivial multiplicative character modulo $p$.
In the paper we consider a problem of obtaining good upper bounds for the exponential sum
\begin{equation}\label{f:def_sum}
    \sum_{a\in A,\, b\in B}\chi(a+b) \,,
\end{equation}
where $A,B$ are arbitrary subsets of the field  $\F_p$.
Exponential sums of such a type were studied by various authors, see e.g. \cite{Chang}--\cite{Kar2}.
There is a well--known hypothesis on sums (\ref{f:def_sum}) which is called the graph Paley conjecture, see the history of the question in \cite{Chang} or  \cite{Shkr_res}, for example.
\begin{Hyp}[Paley graph]
  Let $\delta>0$ be a real number, $A,B\subset\mathbb{F}_p$ be arbitrary sets with  $\abs{A}>p^{\delta}$ and $\abs{B}>p^\delta$. Then there exists a number $\tau=\tau(\delta)$ such that for any sufficiently
 large prime number $p$ and all nontrivial characters $\chi$
 the following holds
\begin{equation}
    \abs{\sum_{a\in A,\, b\in B}\chi(a+b)}<p^{-\tau}\abs{A}\abs{B} \,.
\label{Paley}\end{equation}
\end{Hyp}

Let us say a few words about the name of the hypothesis.
The \emph{Paley graph} is the graph $G(V,E)$ with the vertex set  $V=\mathbb{F}_p$ and the set of edges $E$ such that $(a,b)\in E$ iff $a-b$ is a quadratic residue. To make the graph non--oriented we assume that  $p\equiv1\pmod 4$.
Under these conditions if one puts $B=-A$  in (\ref{Paley})
and  takes $\chi$ equal to the Legendre symbol
then an interesting statement would follow, namely the size of the maximal clique in the Paley graph (as well as its independent number) grows
slowly than   $p^\delta$ for any positive $\delta$.

Unfortunately, at the moment we know few facts about the hypothesis.
 An affirmative answer
 was obtained
 just in the situation  $\abs{A}>p^{\frac12+\delta}$, $\abs{B}>p^{\delta}$, see \cite{Kar}---\cite{Kar2}. Even in the case $\abs{A}\sim\abs{B}\sim p^{\frac12}$ inequality~(\ref{Paley})
  is unknown, see \cite{Kar2}.
  However, nontrivial bounds of sum (\ref{f:def_sum}) can be obtained for structural sets $A$ and $B$ with  weaker restrictions for the sizes of the sets, see~\cite{Chang}, \cite{FI}, \cite{Kar}.
Thus, in paper~\cite{Chang} Mei--Chu Chang proved such an estimate provided one of the sets
 $A$ or $B$ has small sumset.
Recall that the {\it sumset} of two sets $X, Y \subseteq \mathbb{F}_p$ is the set
$$X+Y = \ost{x+y\,:\,x\in X, y\in Y} \,.$$
\begin{Th}[Chang]
Let $A,\,B\subset\mathbb{F}_p$ be arbitrary sets, $\chi$  be a nontrivial multiplicative character modulo $p$
and $K,\delta$ be positive numbers with
\begin{gather*}
\abs{A}>p^{\frac49+\delta},\\
\abs{B}>p^{\frac49+\delta},\\
\abs{B+B}<K|B| \,.
\end{gather*}
Then there exists $\tau=\tau(\delta, K)>0$ such that the inequality
$$\abs{\sum_{a\in A,\, b\in B}\chi(a+b)}<p^{-\tau}\abs{A}\abs{B}$$
holds for all $p>p(\delta, K)$.
\label{t:Chang}
\end{Th}

In paper~\cite{Shkr_Vol} Chang's assumption $\abs{A}>p^{\frac49+\delta},\abs{B}>p^{\frac49+\delta}$ was refined.

\begin{Th}
Let $A,\,B\subset\mathbb{F}_p$ be sets and $K,L,\delta > 0$ be numbers with
\begin{gather*}
\abs{A}>p^{\frac{12}{31}+\delta},\\
\abs{B}>p^{\frac{12}{31}+\delta},\\
\abs{A+A}<K\abs{A},\\
\abs{A+B}<L\abs{B}.
\end{gather*}
Then for any nontrivial multiplicative character $\chi$ modulo $p$ one has
\begin{equation*}
    \abs{\sum_{a\in A,\, b\in B}\chi(a+b)}\ll\sqrt{\frac{L\log 2K}{\delta\log p}} \cdot \abs A\abs B
\end{equation*}
provided $p>p(\delta, K, L)$.
\end{Th}
In our paper we improve previous results and prove the following theorem.

\begin{Th}[Main result]\label{main_theorem}
Let $A,\,B\subset\mathbb{F}_p$ be sets and $K,L,\delta > 0$ be numbers with
\begin{gather}
\label{0}\abs{A},\abs{B}>p^{\frac13+\delta},\\
\label{1}\abs{A+A}<K|A|,\\
\label{-1}\abs{B+B}<L\abs B.
\end{gather}
Then there exists $\tau=\tau(\delta, K)=\delta^2(\log 2K)^{-3+o(1)}$ such that the inequality
$$\abs{\sum_{\substack{a\in A,\,b\in B}}\chi(a+b)}<p^{-\tau}\abs A\abs B$$
holds for all $p>p(\delta, K, L)$. Here $\chi$  is a nontrivial multiplicative character modulo $p$.
\end{Th}

It is interesting that our new approach does not allow to improve any bound for the ternary character sum from~\cite{Shkr_Vol}.

\section*{Definitions and notation}

Recall that the (Minkowski) {\it sumset}  of two sets $A$ and $B$ from the field $\mathbb{F}_p$
is the set
$$A+B = \ost{a+b\,:\,a\in A, b\in B} \,.$$
In a similar way one can define the {\it difference} of two sets $A$ and $B$ as
$$A-B = \ost{a-b\,:\,a\in A, b\in B};$$
Also for an arbitrary $g\in \mathbb{F}_p$ by  $g+A$ denote the sumset $\ost{g}+A$.

Besides, we denote
$$[a,b]=\ost{i\in\mathbb{Z}\,:\,a\le i\le b} \,.$$
Also we will use the
{\it third moment of convolution} of set $A$, see \cite{Schoen_shkr}
$$\E^{\times}_3(A) = \abs{\ost{(a_1,a_2,a_3,a_1',a_2',a_3')\in A^6\,:\,a_1/a_1'=a_2/a_2'=a_3/a_3'}} \,.$$

{\it Generalized arithmetic progression of dimension $d$} is a set $P\subset\mathbb{F}_p$
of the form
\begin{equation}P=a_0+\ost{\sum_{j=1}^dx_ja_j\,:\,x_j\in\floor{0,H_j-1}} \,,\label{representation}\end{equation}
where  $a_0,a_1,\ldots,a_d$ are
some elements from $\mathbb{F}_p$; $P$ is said to be {\it proper} if all of the sums in (\ref{representation})
are distinct (in the case $\abs P=\prod_{j=1}^d H_j$).

\begin{Th}[Freiman]
 For any set $A\subseteq \mathbb{F}_p$ such that $\abs{A+A}\le K\abs A$ there is a generalized arithmetic progression $P$ of dimension $d$ containing  $A$ such that $d\le C(K)$ and $\abs P\le e^{C(K)}\abs A$.
 Here $C(K) >0$ is a constant which depends on $K$ only but not on the set $A$.
\end{Th}

It is known that the constant $C(K)$ can be taken equal $C(K) = (\log 2K)^{3+o(1)}$, see \cite{Sanders_survey}.

Also let us remind that a  multiplicative character $\chi$  modulo $p$ is a homomorphism from $\mathbb{F}^{*}_p$ into
the unit circle of the complex plane.
The character $\chi_0\equiv1$ is called trivial and the  conjugate to a character
 $\chi(x)$ is the character  $\overline{\chi}(x)=\overline{\chi(x)}=\chi(x^{-1})$.
The order of a character $\chi$ is the least positive integer $d$ such that $\chi^d=\chi_0$. One can read about properties of multiplicative characters in \cite{Stepanov} or \cite{IK}.

We need a variant of Andr\'{e} Weil's result (see  Theorem~11.23 in~\cite{IK}).

\begin{Th}[Weil]
Let $\chi$ be a nontrivial multiplicative character modulo $p$ of order $d$.
Suppose that a polynomial $f$ has $m$ distinct roots and there is no polynomial $g$ such that  $f=g^d$.
Then
$$\abs{\sum_{x\in\mathbb{F}_p}\chi\skob{f(x)}}\le(m-1)\sqrt p \,.$$
\end{Th}

Also we will use the H\"{o}lder inequality.

\begin{Lemma}[The H\"{o}lder inequality]
For any positive $p$ and $q$ such that  $\frac1p+\frac1q=1$ one has
$$\abs{\sum_{k=1}^{n} x_ky_k}\le\skob{\sum_{k=1}^n\abs{x_k}^p}^{\frac1p}\skob{\sum_{k=1}^n\abs{y_k}^q}^{\frac1q} \,.$$
In particular, we have the Cauchy--Schwarz inequality
$$\skob{\sum_{k=1}^{n}x_ky_k}^2\le\skob{\sum_{k=1}^nx_k^2}\skob{\sum_{k=1}^ny_k^2} \,.$$
\end{Lemma}

\section*{Some preliminary lemmas}

In paper~\cite{Shkr2} the following result was proved.

\begin{Th}\label{third_energy}
Let $A\subset\F_p$. Suppose that $|A|^{11}|A+A|\le p^8$. Then
$$\E^{\times}_3(A)\ll\frac{|A+A|^{15/4}}{|A|^{3/4}}\log{|A|}.$$
\end{Th}

\begin{Lemma}\label{system_solution}
Suppose that
$A,\,B\subset\mathbb{F}_p$ are any sets and $K,L$ are positive numbers such that
\begin{gather*}
\abs{A},\abs{B} < \sqrt p,\\
\abs{A+A}<K|A|,\\
\abs{B+B}<L\abs B.
\end{gather*}
Then  the system of equations
\begin{equation}\begin{cases}
\frac{b_1}a=\frac{b_1'}{a'}\\
\frac{b_2}a=\frac{b_2'}{a'}
\end{cases}
\label{system}
\end{equation}
has
\begin{equation}\label{f:bound_E^t_3}
    O(K^{5/4}L^{5/2}|A||B|^2\log p
        + |A|^2 |B|)
\end{equation}
solutions in the variables $(a,a',b_1,b_1',b_2,b_2')\in A^2\times B^4$.
\end{Lemma}
\begin{proof}
Clearly, the number of trivial solutions, where $b_1=b_1'=b_2=b_2'=0$ and $a_1,a_2\in A$ are any numbers does not exceed $\abs{A}^2$. Moreover the number of solutions, where $b_1=b_1'=0$ and $b_2'=\frac{a'b_2}{a}$ does not exceed $\abs{A}^2\abs B$ and this gives us the second term in (\ref{f:bound_E^t_3}).
Below we will assume that all numerators in (\ref{system}) are nonzero.

For any $\lambda\in\mathbb{F}_p$ put
\begin{gather*}
f(\lambda) = \abs{\ost{(a,a')\in A^2\,:\,\lambda = \frac a{a'}}} \,,\\
g(\lambda) = \abs{\ost{(b,b')\in B^2\,:\,\lambda = \frac b{b'}}} \,.
\end{gather*}

Further, the systems of the equations (\ref{system}) can be rewritten in an equivalent form, namely,
$$\frac{a}{a'}=\frac{b_1}{b_1'}=\frac{b_2}{b_2'} \,.$$
Whence the number of its solutions equals $\sum\limits_{\lambda\in\mathbb{F}_p}f(\lambda)g(\lambda)^2$. Estimating this sum by the H\"{o}lder inequality and applying Theorem~\ref{third_energy}, we complete the proof of the Lemma
\begin{multline*}
\sum_{\lambda\in\mathbb{F}_p}f(\lambda)g(\lambda)^2\le\skob{\sum_{\lambda\in\mathbb{F}_p}f(\lambda)^3}^{\frac13}\skob{\sum_{\lambda\in\mathbb{F}_p}g(\lambda)^3}^{\frac23}=(\E^{\times}_3(A))^{\frac13}(\E^{\times}_3(B))^{\frac23}\ll\\ \ll\left(\frac{|A+A|^{15/4}}{|A|^{3/4}}\log{|A|}\right)^{\frac13}\left(\frac{|B+B|^{15/4}}{|B|^{3/4}}\log{|B|}\right)^{\frac23}\le K^{5/4}L^{5/2}|A||B|^2\log p.
\end{multline*}
The conditions of Theorem~\ref{third_energy} are met, as $|A|^{11}|A+A|\le p^{\frac{11}2+1}\le p^8$ and $|B|^{11}|B+B|\le p^{\frac{11}2+1}\le p^8$.
\end{proof}

Weil's Theorem implies the following result, see Lemma~14~in paper~\cite{Shkr_Vol}.

\begin{Lemma}\label{davenport}
For any nontrivial character $\chi$, an arbitrary  set  $I\subset\mathbb{F}_p$ and a positive integer $r$ one has
$$\sum_{u_1,u_2\in \mathbb{F}_p}\abs{\sum_{t\in I}\chi(u_1+t)\overline{\chi}(u_2+t)}^{2r}< p^2\abs{I}^rr^{2r}+4r^2p\abs{I}^{2r} \,.$$
\end{Lemma}

\section*{The proof of the main result}

\begin{proof}[The proof of Theorem~\ref{main_theorem}]
The beginning of the proof is similar to the arguments from~\cite{Shkr_Vol}. We will assume that $|A|, |B| < \sqrt{p}$. According the Freiman structural theorem on sets with small doubling
there is a generalized arithmetic progression $A_1=a_0+P\subseteq \mathbb{F}_p$ of the dimension $d$, where $$P=\ost{\sum_{j=1}^dx_ja_j\,:\,x_j\in\floor{0,H_j-1}}$$
such that
$$A\subset A_1$$
$$d\le C(K)$$
$$\abs{A_1}<e^{C(K)}\abs{A} \,.$$
Put
$$\alpha=\frac{3\delta}{4d},\quad
r=\left\lceil\frac1\alpha \right\rceil\,.$$
Take the interval $I=\floor{1,p^\alpha}$ and the generalized progression $A_0$ of the dimension $d$ defined as
$$A_0=\ost{\sum_{j=1}^dx_ja_j\,:\,x_j\in\floor{0,p^{-2\alpha}H_j}} \,.$$
Clearly,
\begin{equation}
\abs{A_0}\ge p^{-2d\alpha}\abs{A_1}\ge p^{-2d\alpha}\abs{A}
\label{2}
\end{equation}
and
\begin{equation}
\abs{A_0+A_0}\le 2^d\abs{A_0} \,. \label{A0+A0}
\end{equation}
Because $A_0I\subseteq \ost{\sum\limits_{j=1}^dx_ja_j\,:\,x_j\in\floor{0,p^{-\alpha}H_j}}$ and hence
\begin{equation*}
    A-A_0I \subseteq \ost{\sum\limits_{j=1}^dx_ja_j\,:\,x_j\in\floor{-p^{-\alpha} H_j, H_j}}
\end{equation*}
we, clearly, get
\begin{equation}\label{A-A_0I}
    \abs{A-A_0I}\le\skob{1+p^{-\alpha}}^d\abs{A_1}\le e^{C(K)}\skob{1+p^{-\alpha}}^d\abs{A}\le e^{C(K)}2^d\abs{A} \,.
\end{equation}

Let us fix $x\in A_0, y\in I$ and estimate the sum
\begin{multline}\abs{\sum_{a\in A,\, b\in B}\chi(a+b)}\le\sum_{a\in A}\abs{\sum_{b\in B}\chi(a+b)}=\\ =\sum_{a\in A-xy}\abs{\sum_{b\in B}\chi(a+b+xy)}\le\sum_{a\in A-A_0I}\abs{\sum_{b\in B}\chi(a+b+xy)}.\label{9}\end{multline}
The numbers $x\in A_0$, $y\in I$ can be taken in such a way that the last sum in~(\ref{9}) does not exceed the mean, whence
\begin{equation}\abs{\sum_{a\in A,\, b\in B}\chi(a+b)}\le\frac1{\abs{A_0}\abs{I}}\sum_{\substack{a\in A-A_0I,\\ x\in A_0,\, y\in I}}\abs{\sum_{b\in B}\chi(a+b+xy)}.
\label{rand}\end{equation}
Now having  any fixed $a\in A-A_0I$, let us estimate the sum $$\sum_{x\in A_0,\, y\in I}\abs{\sum_{b\in B}\chi(a+b+xy)}=\sum_{x\in A_0,\, y\in I}\abs{\sum_{b\in B_a}\chi(b+xy)} \,.$$
Here we have denoted $B_a=a+B$. By the Cauchy--Schwarz inequality, we get
\begin{multline}\skob{\sum_{x\in A_0,\, y\in I}\abs{\sum_{b\in B_a}\chi(b+xy)}}^2\le\\ \le\skob{\sum_{x\in A_0,\, y\in I}1}\skob{\sum_{x\in A_0,\, y\in I}\abs{\sum_{b\in B_a}\chi(b+xy)}^2}=\\ =\abs{A_0}\abs{I}\skob{\sum_{\substack{x\in A_0,\, y\in I,\\ b_1,\, b_2\in B_a}}\chi(b_1+xy)\overline{\chi}(b_2+xy)}.
\label{kbsh}
\end{multline}
For any pair $(u_1,u_2)\in\mathbb{F}^2_p$ put
$$\nu(u_1,u_2)=\abs{\ost{(b_1,b_2,x)\in B_a^2\times A_0\,:\,\frac{b_1}x=u_1
\mbox{ Х } \frac{b_2}x=u_2}} \,.$$

Then for any  $x\neq 0$, we have
\begin{multline}
\sum_{\substack{x\in A_0,\, y\in I,\\ b_1,b_2\in B_a}}\chi(b_1+xy)\overline{\chi}(b_2+xy)=\\
=\sum_{\substack{x\in A_0,\, y\in I,\\ b_1,b_2\in B_a}}\chi(b_1x^{-1}+y)\overline{\chi}(b_2x^{-1}+y)=\\
=\sum_{u_1,u_2\in\mathbb{F}_p^2}\nu(u_1,u_2)\sum_{y\in I}\chi(u_1+y)\overline{\chi}(u_2+y)\le\\
\le\skob{\sum_{u_1,u_2}\nu(u_1,u_2)}^{1-\frac1r} \skob{\sum_{u_1,u_2}\nu(u_1,u_2)^2}^{\frac{1}{2r}} \times\\ \times \skob{\sum_{u_1,u_2}\abs{\sum_{t\in I}\chi(u_1+t)\overline{\chi}(u_2+t)}^{2r}}^{\frac1{2r}}.
\label{main_esteem}
\end{multline}
The inequality in~(\ref{main_esteem}) follows from the H\"{o}lder inequality  and the Cauchy--Schwarz inequality.
By Lemma~\ref{davenport}
\begin{multline}\skob{\sum_{u_1,u_2}\abs{\sum_{t\in I}\chi(u_1+t)\overline{\chi}(u_2+t)}^{2r}}^{\frac1{2r}}<\skob{p^2\abs{I}^rr^{2r}+4r^2p\abs{I}^{2r}}^{\frac1{2r}}\le\\ \le r\abs{I}^{\frac12}p^{\frac1r}+(2r)^{\frac1r}p^{\frac1{2r}}\abs{I}\le2rp^{\frac1{2r}}\abs{I}\,.\label{3.22}\end{multline}
The last inequality takes place because $\abs{I}\ge p^{\frac1{r}}$ and $r\ge 2$.
Further note that
\begin{equation}
\sum_{u_1,u_2\in \mathbb{F}_p}\nu(u_1,u_2)=\abs{B}^2\abs{A_0},
\label{3.19}
\end{equation}
and by Lemma~\ref{system_solution}, combining with inequalities (\ref{0}), (\ref{1}), (\ref{-1}), (\ref{2}) and (\ref{A0+A0}),  we obtain
\begin{multline}
\sum_{u_1,u_2\in \mathbb{F}_p}\nu(u_1,u_2)^2=\\
=\abs{\ost{(x,x',b_1,b_1',b_2,b_2')\in A_0^2\times B_a^4\,:\,\frac{b_i}x=\frac{b'_i}{x'}\text{ ДКЪ $i=1,2$}}}\ll\\
\ll 2^{\frac{5d}4}L^{\frac52}\abs{A_0}\abs{B_a}^{2}\log p
    + |A_0|^2 |B_a| \ll\\
\ll \skob{\abs{A_0}\abs B^2}^{2}\skob{2^{\frac{5d}4}L^{\frac52}\abs{A_0}^{-1}\abs B^{-2}\log p + \abs{B}^{-3}}\ll\\
\ll\skob{\abs{A_0}\abs B^2}^{2}2^{\frac{5d}4}L^{\frac52}p^{2d\alpha-3\skob{\frac13+\delta}}\log p=\\
= \skob{\abs{A_0}\abs B^2}^{2}2^{\frac{5d}4}L^{\frac52}p^{2d\alpha-3\delta-1}\log p \,.
\label{3.20}
\end{multline}
Using estimates~(\ref{kbsh})---(\ref{3.20}),
we see that
\begin{equation*}
\skob{\sum_{x\in A_0, y\in I}\abs{\sum_{b\in B_a}\chi(b+xy)}}^2\ll\skob{\abs{A_0}\abs I\abs B}^2r2^{\frac {5d}{8r}} L^{\frac{5}{4r}}p^{\frac{d\alpha}{r}-\frac{3\delta}{2r}}\log^{\frac1{2r}} p
    \,.
\end{equation*}
Because
$\alpha = \frac{3\delta}{4 d}$ and hence $r\ge \frac1\alpha = \frac{4d}{3\delta}$ and $r\le \frac1\alpha+1=\frac{4d+3\delta}{3\delta}\le \frac{5d}{3\delta}$, we obtain further
\begin{equation}
\skob{\sum_{x\in A_0, y\in I}\abs{\sum_{b\in B_a}\chi(b+xy)}}^2
\ll\skob{\abs{A_0}\abs I\abs B}^2\frac d\delta L^{\frac{15\delta}{16d}}p^{-\frac{3\delta}{4r}}\log^{\frac1{2r}} p \,.
\label{3.24}
\end{equation}
Bound~(\ref{3.24}) takes place for any $a$ and thus inequalities (\ref{A-A_0I}), (\ref{rand})
imply
\begin{multline}
\abs{\sum_{a\in A,b\in B}\chi(a+b)}\ll \sqrt{\frac d\delta} L^{\frac{15\delta}{32}}p^{-\frac{3\delta}{8r}}\abs{A-A_0I}\abs{B}\log^{\frac1{4r}} p\ll\\ \ll \sqrt{\frac d\delta} L^{\frac{15\delta}{32d}}2^de^{C(K)}p^{-\frac{9\delta^2}{40d}}\abs A\abs{B}\log^{\frac1{4r}} p \,.
\label{final_inequality}
\end{multline}
The theorem follows from~(\ref{final_inequality}) if one takes $\tau=\frac{\delta^2}{100 C(K)}$, for example.
\end{proof}

\begin{note}
From inequality (\ref{final_inequality}) it is easy to find the quantity $p(\delta, K,L)$ in a precise way.
Indeed, it is enough to choose $p$ such that
$\log p \gg \frac{C^2(K)}{\delta^2}$, $\log p \gg \frac{C(K)}{\delta^2}\log{\frac1\delta}$ and $\log p \gg \frac{C(K) \log L }{\delta}$.
It shows that we have subexponential dependence of the constants $K, L$ on $p$ in our theorem.
\label{p(d,K,L)}
\end{note}

\noindent{A.S.~Volostnov\\
Steklov Mathematical Institute of Russian Academy of Sciences,\\
ul. Gubkina, 8, Moscow, Russia, 119991.}\\
{\tt gyololo@rambler.ru}

\end{document}